\documentclass[11pt]{amsart}
\usepackage{mathrsfs}
\usepackage{amsfonts}
\usepackage{latexsym,amsmath,amssymb}

 \renewcommand{\epsilon}{\varepsilon}

\newtheorem{theorem}{Theorem}[section]

 \newtheorem{lemma}[theorem]{Lemma}
 
 \newtheorem{Corollary}[theorem]{Corollary}
 \newtheorem{proposition}[theorem]{proposition}
 \newtheorem{Proposition}[theorem]{Proposition}
\newtheorem{deff}[theorem]{Definition}
 
 \newcommand{\bth}{\begin{theorem}}
 \newcommand{\ble}{\begin{lemma}}
 \newcommand{\bcor}{\begin{corr}}
 \newcommand{\bdeff}{\begin{deff}}
 \newcommand{\bprop}{\begin{proposition}}
 \newcommand{\ele}{\end{lemma}}
 \newcommand{\ecor}{\end{corr}}
 \newcommand{\edeff}{\end{deff}}
 
 \newcommand{\eprop}{\end{proposition}}

 \renewcommand{\Pi}{\varPi}

 \renewcommand{\epsilon}{\varepsilon}

\numberwithin{equation}{section}

\thanks{The second author is supported  by National Science Foundation of China(No.11771103)
and Guangxi Natural Science Foundation (No.2017GXNSFFA198013).
The third author is supported by Guangdong Natural Science Foundation(No.2016A030307008).}

\title
[Nonlinear second boundary conditions ]{On the second boundary value problem for  a class of fully nonlinear flows II }
\author{Juanjuan Chen}
\address{Faculty of Information Technology, Macau University of Science and Technology, Macau, E-mail: janehappy@gxnu.edu.cn}
\author{Rongli Huang}
\address{School of Mathematics and Statistics, Guangxi Normal University,
Guilin, Guangxi 541004, People's Republic of China,
 E-mail: ronglihuangmath@gxnu.edu.cn}
\author{Yunhua Ye}
\address{School of Mathematics, Jiaying University,
Meizhou, Guangdong 514015, People's Republic of China, E-mail: mathyhye@163.com}
\date{}

\begin{document}
\maketitle
\begin{abstract}
This article is a continuation of earlier work [R.L. Huang and Y.H. Ye,  On the second boundary value problem for  a class of fully nonlinear flows I, to appear in International Mathematics Research Notices], where the long time existence and convergence  were given on some general parabolic type special Lagrangian equations.
The long time existence and convergence of the flow had been obtained in all cases.
In particular, we can prescribe the second boundary value problems for a family of special Lagrangian graphs.
\end{abstract}

\let\thefootnote\relax

\section{Introduction}

In this article, we discuss the existence of a family of special Lagrangian graphs by solving the corresponding special Lagrangian equations.
Let $\Omega$, $\tilde{\Omega}$ be two uniformly convex bounded
domains with smooth boundary in  $\mathbb{R}^{n}$  and  $a=\cot\tau$, $b=\sqrt{|\cot^{2}\tau-1|}$ for $\tau \in [0,\frac{\pi}{2}]$.
Here we consider the minimal Lagrangian diffeomorphism problem \cite{RS} which  is equivalent to
the following fully nonlinear elliptic equations with second boundary condition (cf. \cite{SM}, \cite{HRO} and \cite{MW}):
\begin{equation}\label{e1.1}
\left\{ \begin{aligned}F_{\tau}(D^{2}u)&=c,
&  x\in \Omega, \\
Du(\Omega)&=\tilde{\Omega},
\end{aligned} \right.
\end{equation}
where
$$F_{\tau}(A)=\left\{ \begin{aligned}
&\sum_{i}\ln\lambda_{i} , \quad
\qquad\quad\qquad\qquad\quad\quad\qquad\quad  \tau=0, \\
& \sum_{i}\ln(\frac{\lambda_{i}+a-b}{\lambda_{i}+a+b}) \qquad\qquad\quad\quad\qquad\quad 0<\tau<\frac{\pi}{4},\\
& -\sum_{i}\frac{1}{1+\lambda_{i}}, \qquad \qquad\qquad \qquad\qquad \qquad\tau=\frac{\pi}{4},\\
& \sum_{i}\arctan(\frac{\lambda_{i}+a-b}{\lambda_{i}+a+b}), \qquad \qquad\qquad \quad\frac{\pi}{4}<\tau<\frac{\pi}{2},\\
& \sum_{i}\arctan\lambda_{i}, \,\,\quad \qquad\qquad \qquad\qquad \qquad\tau=\frac{\pi}{2},
\end{aligned} \right.$$
$\lambda_i ~(1\leq i\leq n)$ are the eigenvalues of the Hessian $D^2 u$ and $Du$ are diffeomorphism from $\Omega$ to $\tilde{\Omega}$.

Our motivation of studying equation \eqref{e1.1} is that they have geometric meanings
which were studied by Warren. To illustrate this, let us recall the definition of Lagrangian and
special Lagrangian graph as in \cite{HRY} or \cite{MW}.
The graph $\Sigma=\{(x,f(x)): x \in \Omega\}$ is Lagrangian if and only if there exists a function $u:\Omega\rightarrow \mathbb{R}$
such that $f(x)=Du(x)$.
Let $\delta_{0}$ be the standard Euclidean metric on $\mathbb{C}^n\cong \mathbb{R}^n\times \mathbb{R}^n$
and $g_{0}$ be the metric defined by
$$2dxdy= \sum_{i}(dx_{i}\bigotimes dy_{i}+dy_{i}\bigotimes dx_{i}).$$
By taking linear combinations of the metrics $\delta_{0}$ and $g_{0}$, Warren \cite{MW} constructed a family of metrics
on  $\mathbb{R}^{n}\times\mathbb{R}^{n}$ for $0\leq \tau \leq \frac{\pi}{2}$:
\begin{equation*}
g_{\tau}=\cos \tau g_{0}+\sin \tau \delta_{0}.
\end{equation*}
With this family of metrics, Warren \cite{MW} derived that the solutions of special Lagrangian equations \eqref{e1.1} correspond to
a family of extremal Lagrangian surfaces.
Warren also studied the extremal volume property of these special Lagrangian graphs in \cite{MW}. For $t<\frac{\pi}{4}$, $M_t=(\mathbb{R}^n\times \mathbb{R}^n,g_t)$ is a pseudo-Euclidean space of index $n$. For $t>\frac{\pi}{4}$, $M_t$ is a Euclidean space. For $t=\frac{\pi}{4}$, $M_t$ carries a degenerate metric of rank $k$.

We have the following definition of special Lagrangian graph as in \cite{MW}.
\begin{deff}\label{d1.1}
 We say that $\Sigma=\{(x,f(x))|x\in \Omega\}$ is a special Lagrangian graph in $(\mathbb{R}^{n}\times\mathbb{R}^{n}, g_{\tau})$ if $$f=Du$$
  and $u$ satisfies
 $$F_{\tau}(D^{2}u (x))=c,\quad x\in \Omega.$$
\end{deff}

Special Lagrangian graphs have attracted considerable interest in recent years and we recall some work concerning equation \eqref{e1.1}
with second boundary conditions. For the case $\tau=0$, in 1991, Delano\"{e} \cite{P} studied the first work on the problem where the dimension is 2 and he obtained a  unique smooth solution.
Later Caffarelli \cite{L} and Urbas \cite{JU} gave the generalization of Delano\"{e}'s theorem to higher dimensions. Using the parabolic methods, Schn$\ddot{\text{u}}$rer and Smoczyk \cite{OK} also
obtained the existence of solutions to (\ref{e1.1}) for $\tau=0$.
As far as $\tau=\frac{\pi}{2}$ is concerned, Brendle and Warren \cite{SM} proved the existence and uniqueness of the solution
by the elliptic methods and the second author \cite{HR} obtained the existence of solution by the parabolic methods. For the case $\tau=\frac{\pi}{4}$,
the existence result of the above problem \eqref{e1.1} was established by the second author and his coauthors by using both
elliptic method \cite{HRO} and parabolic methods \cite{HRY}.

This article is devoted to studying the equation (\ref{e1.1}) for all
$$\tau\in(0,\frac{\pi}{4})\cup(\frac{\pi}{4},\frac{\pi}{2}).$$
As in \cite{HR} and \cite{HRY}, we consider the corresponding parabolic type special Lagrangian equations and use parabolic methods to solve problem \eqref{e1.1}.
We settle the longtime existence and convergence of smooth solutions
for the following second boundary value problem to parabolic type special Lagrangian equations
\begin{equation}\label{e1.2}
\left\{ \begin{aligned}\frac{\partial u}{\partial t}&=F_{\tau}(D^{2}u),
& t>0,\quad x\in \Omega, \\
Du(\Omega)&=\tilde{\Omega}, &t>0,\qquad\qquad\\
 u&=u_{0}, & t=0,\quad x\in \Omega.
\end{aligned} \right.
\end{equation}

We will prove the solutions of the above special Lagrangian equations will converge to those of problem \eqref{e1.1}. 
Our main results of the present article are summarized as follows.
\begin{theorem}\label{t1.1}
Assume that $\Omega$, $\tilde{\Omega}$ are bounded, uniformly convex domains with smooth boundary in $\mathbb{R}^{n}$, $0<\alpha_{0}<1$,
$\tau\in(0,\frac{\pi}{4})\cup(\frac{\pi}{4},\frac{\pi}{2})$.
  Then for any given initial function $u_{0}\in C^{2+\alpha_{0}}(\bar{\Omega})$
  which is   uniformly convex and satisfies $Du_{0}(\Omega)=\tilde{\Omega}$,  the  strictly convex solution of (\ref{e1.2}) exists
  for all $t\geq 0$ and $u(\cdot,t)$ converges to a function $u^{\infty}(x,t)=u^\infty(x)+C_{\infty}\cdot t$ in $C^{1+\zeta}(\bar{\Omega})\cap C^{4+\alpha}(\bar{D})$ as $t\rightarrow\infty$
  for any $D\subset\subset\Omega$, $\zeta<1$,$0<\alpha<\alpha_{0}$, that is,
 $$\lim_{t\rightarrow+\infty}\|u(\cdot,t)-u^{\infty}(\cdot,t)\|_{C^{1+\zeta}(\bar{\Omega})}=0,\qquad
  \lim_{t\rightarrow+\infty}\|u(\cdot,t)-u^{\infty}(\cdot,t)\|_{C^{4+\alpha}(\bar{D})}=0.$$
And $u^{\infty}(x)\in C^{\infty}(\bar{\Omega})$ is a solution of
\begin{equation}\label{e1.3}
\left\{ \begin{aligned}F_{\tau}(D^{2}u)&=C_{\infty},
&  x\in \Omega, \\
Du(\Omega)&=\tilde{\Omega}.
\end{aligned} \right.
\end{equation}
The constant $C_{\infty}$ depends only on $\Omega$, $\tilde{\Omega}$ and $F$. The solution to (\ref{e1.3}) is unique up to additions of constants.
\end{theorem}

Combining Definition \ref{d1.1}, Theorem \ref{t1.1} with the  results for $\tau=0, \frac{\pi}{4}, \frac{\pi}{2}$,   we can extend Brendle-Warren's  theorem \cite{SM} to the following:
\begin{Corollary}\label{c1.1}
Let $\Omega$, $\tilde{\Omega}$ be bounded, uniformly convex domains with smooth boundary in $\mathbb{R}^{n}$ and $0\leq\tau\leq\frac{\pi}{2}$. Then
there exists a diffeomorphism f: $\Omega\rightarrow\tilde{\Omega}$ such that
 $$\Sigma=\{(x,f(x))|x\in \Omega\}$$
is a special Lagrangian graph in $(\mathbb{R}^{n}\times\mathbb{R}^{n}, g_{\tau})$.
\end{Corollary}
The rest of this article is organized as follows.
 In Section 2, we present a preliminary result. The result will be used  to give the proof of the main theorem in the article.
In Section 3, we prove that the structure conditions in (\ref{e1.2}) satisfy  the hypotheses in Proposition \ref{p1.1}.
Therefore we are able to characterize the longtime behavior of
parabolic type special Lagrangian equation (\ref{e1.2}) and give
the proof of Theorem \ref{t1.1}.

\section{Preliminaries}
Consider
\begin{equation}\label{e2.1}
\left\{ \begin{aligned}\frac{\partial u}{\partial t}&=F(D^{2}u),
& t>0,\quad x\in \Omega, \\
Du(\Omega)&=\tilde{\Omega}, &t>0,\qquad\qquad\\
 u&=u_{0}, & t=0,\quad x\in \Omega,
\end{aligned} \right.
\end{equation}
where $F$ is a $C^{2+\alpha_{0}}$ function for some $0<\alpha_{0}<1$ defined on the cone 
$\Gamma_{+}$ of positive definite symmetric matrices, which is monotonically increasing and
\begin{equation}\label{e1.4aa}
\left\{ \begin{aligned}&F[A]:=F(\lambda_{1},\lambda_{2},\cdots, \lambda_{n})\\
&F(\cdots,\lambda_{i},\cdots,\lambda_{j},\cdots )=F(\cdots,\lambda_{j},\cdots,\lambda_{i},\cdots),\quad \text{for} \,\,\, 1\leq i<j\leq n,
\end{aligned} \right.
\end{equation}
with
$$\lambda_{1}\leq\lambda_{2}\leq\cdots\leq\lambda_{n}$$
being the eigenvalues of the $n\times n$ symmetric matrix $A$.

For any $\mu_{1}>0, \mu_{2}>0$, we define
$$\Gamma^{+}_{]\mu_{1},\mu_{2}[}=\{(\lambda_{1},\lambda_{2},\cdots, \lambda_{n})|0\leq\lambda_{1}\leq\lambda_{2}\leq\cdots\leq\lambda_{n}, \lambda_{1}\leq \mu_{1}, \lambda_{n}\geq \mu_{2}\}.$$
We assume that there exist  positive constants $\lambda, \Lambda$ depending only on $\mu_{1}, \mu_{2}$ such that
for any $(\lambda_{1},\lambda_{2},\cdots, \lambda_{n})\in \Gamma^{+}_{]\mu_{1},\mu_{2}[}$:
\begin{equation}\label{e1.16}
 \Lambda\geq\sum^{n}_{i=1}\frac{\partial F}{\partial \lambda_{i}}\geq \lambda,
\end{equation}
\begin{equation}\label{e1.17}
  \Lambda\geq\sum^{n}_{i=1}\frac{\partial F}{\partial \lambda_{i}}\lambda^{2}_{i}\geq \lambda.
\end{equation}
In addition,
\begin{equation}\label{e1.15}
 F(A)\,\, and\,\,F^{*}(A)\triangleq-F(A^{-1})\,\,are\,\, concave\,\, on\,\,\Gamma_{+}.
\end{equation}
Moreover, we assume that there exist two functions $f_{1}$, $f_{2}$ which are monotonically increasing in $(0,+\infty)$ satisfying
\begin{equation}\label{e1.15a}
 f_{1}(\lambda_{1})\leq F(\lambda_{1},\lambda_{2},\cdots, \lambda_{n})\leq f_{2}(\lambda_{n})\quad (\forall\,\,\, 0\leq\lambda_{1}\leq\lambda_{2}\leq\cdots\leq\lambda_{n}),
\end{equation}
and for any $\Phi, \Psi \in \pounds$,
\begin{equation}\label{e1.15b}
\left\{ \begin{aligned}
 f_{1}(t)\leq \Phi\Rightarrow \exists t_{1}>0,\,\, t\leq  t_{1} ,\\
 f_{2}(t)\geq \Psi\Rightarrow \exists t_{2}>0,\,\, t\geq t_{2} ,
 \end{aligned} \right.
\end{equation}
where
$$\pounds=\{\Upsilon| \exists(\lambda_{1},\lambda_{2},\cdots, \lambda_{n}),  0<\lambda_{1}\leq\lambda_{2}\leq\cdots\leq\lambda_{n}, \Upsilon=F(\lambda_{1},\lambda_{2},\cdots, \lambda_{n})\}.$$

The following proposition concerning convergence of general uniformly parabolic operators under
certain structural conditions
 plays a fundamental role in our proof of Theorem
\ref{t1.1}.
\begin{Proposition}(R.L. Huang and Y.H. Ye, see Theorem 1.1 in  \cite{HRY}.)\label{p1.1}
Assume that $\Omega$, $\tilde{\Omega}$ are bounded, uniformly convex domains with smooth boundary in $\mathbb{R}^{n}$, $0<\alpha_{0}<1$ and
the map $F$ satisfies (\ref{e1.4aa}),  (\ref{e1.16}), (\ref{e1.17}), (\ref{e1.15}), (\ref{e1.15a}), (\ref{e1.15b}).
  Then for any given initial function $u_{0}\in C^{2+\alpha_{0}}(\bar{\Omega})$
  which is   uniformly convex and satisfies $Du_{0}(\Omega)=\tilde{\Omega}$,  the  strictly convex solution of (\ref{e2.1}) exists
  for all $t\geq 0$ and $u(\cdot,t)$ converges to a function $u^{\infty}(x,t)=u^\infty(x)+C_{\infty}\cdot t$ in $C^{1+\zeta}(\bar{\Omega})\cap C^{4+\alpha}(\bar{D})$ as $t\rightarrow\infty$
  for any $D\subset\subset\Omega$, $\zeta<1$,$0<\alpha<\alpha_{0}$, that is,

  $$\lim_{t\rightarrow+\infty}\|u(\cdot,t)-u^{\infty}(\cdot,t)\|_{C^{1+\zeta}(\bar{\Omega})}=0,\qquad
  \lim_{t\rightarrow+\infty}\|u(\cdot,t)-u^{\infty}(\cdot,t)\|_{C^{4+\alpha}(\bar{D})}=0.$$
And $u^{\infty}(x)\in C^{1+1}(\bar{\Omega})\cap C^{4+\alpha}(\Omega)$ is a solution of
\begin{equation}\label{e1.18}
\left\{ \begin{aligned}F(D^{2}u)&=C_{\infty},
&  x\in \Omega, \\
Du(\Omega)&=\tilde{\Omega}.
\end{aligned} \right.
\end{equation}
The constant $C_{\infty}$ depends only on $\Omega$, $\tilde{\Omega}$ and $F$. The solution to (\ref{e1.18}) is unique up to additions of constants.
\end{Proposition}

 In \cite{HRY}, Huang and Ye first used the inverse function theorem to establish the short time existence of the flow \eqref{e2.1}. 
 Then the authors used structural conditions \eqref{e1.4aa} - \eqref{e1.15b} to construct suitable auxiliary functions as barriers 
 and finally established the apriori estimates needed to prove the  convergence of the flow.

\section{Proof of the main thoerem}
In this article, we verify that the hypotheses in (\ref{e1.4aa})-(\ref{e1.15b}) are valid
for the geometric evolution equation (\ref{e1.2}) via elementary methods. To that end,
we require an elementary result for monotone increasing function.
\begin{lemma}\label{l3.1}
Let $f(t)$ is  monotone increasing continuous function on $(0,+\infty)$.
Then for any $0<\lambda_{1}\leq\lambda_{2}\leq\cdots\leq\lambda_{n}$,
there exists a unique $\lambda \in [\lambda_{1},\lambda_{n}]$, such that
$$f(\lambda)=\frac{\sum^{n}_{i=1}f(\lambda_{i})}{n}.$$
\end{lemma}
\begin{proof}
Since  $f(t)$ is  monotone increasing, then
$$f(\lambda_{1})\leq\frac{\sum^{n}_{i=1}f(\lambda_{i})}{n}\leq f(\lambda_{n}).$$
By making use of the intermediate value theorem of continuous functions, we obtain the conclusion.
\end{proof}

We put
$$F_{\tau}(\lambda_{1},\lambda_{2},\cdots, \lambda_{n})=\left\{ \begin{aligned}
& \sum_{i}\ln(\frac{\lambda_{i}+a-b}{\lambda_{i}+a+b}) \qquad\qquad\quad\quad\qquad\quad 0<\tau<\frac{\pi}{4},\\
& \sum_{i}\arctan(\frac{\lambda_{i}+a-b}{\lambda_{i}+a+b}), \qquad \qquad\qquad \quad\frac{\pi}{4}<\tau<\frac{\pi}{2}.
\end{aligned} \right.$$
Without loss of generality, we always assume that $0<\lambda_{1}\leq\lambda_{2}\leq\cdots\leq\lambda_{n}$.
\begin{lemma}\label{l3.2}
For any $\tau\in(0,\frac{\pi}{4})\cup(\frac{\pi}{4},\frac{\pi}{2})$, the operator $F_{\tau}$ satisfies  the hypotheses in (\ref{e1.4aa})-(\ref{e1.15b}).
\end{lemma}
\begin{proof}
Case 1, $\tau\in(0,\frac{\pi}{4})$.

It's obvious that $a=\cot\tau>b=\sqrt{|\cot^{2}\tau-1|}$. We observe
$$\frac{\partial F_{\tau}}{\partial\lambda_{i}}=\frac{1}{\lambda_{i}+a-b}-\frac{1}{\lambda_{i}+a+b}=\frac{2b}{(\lambda_i+a-b)(\lambda_i+a+b)}>0.$$
Then the equation (\ref{e2.1}) is parabolic and  $F_{\tau}$ satisfies (\ref{e1.4aa}).
For any $\mu_{1}>0, \mu_{2}>0$,  if
$\lambda_{1}\leq \mu_{1}, \lambda_{n}\geq \mu_{2},$ then we obtain
\begin{equation}\label{e3.1}
\begin{aligned}
\frac{2nb}{(a-b)(a+b)}&\geq\sum^{n}_{i=1}\frac{\partial F_{\tau}}{\partial\lambda_{i}}=\sum^{n}_{i=1}\frac{2b}{(\lambda_{i}+a-b)(\lambda_{i}+a+b)}\\
&\geq \frac{2b}{(\lambda_{1}+a-b)(\lambda_{1}+a+b)}\\
&\geq  \frac{2b}{(\mu_{1}+a-b)(\mu_{1}+a+b)}
\end{aligned}
\end{equation}
and
\begin{equation}\label{e3.2}
\begin{aligned}
2nb&\geq\sum^{n}_{i=1}\frac{\partial F_{\tau}}{\partial\lambda_{i}}\lambda^{2}_{i}=\sum^{n}_{i=1}\frac{2b\lambda^{2}_{i}}{(\lambda_{i}+a-b)(\lambda_{i}+a+b)}\\
&\geq \frac{2b\lambda^{2}_{n}}{(\lambda_{n}+a-b)(\lambda_{n}+a+b)}\\
&\geq  \frac{2b\mu^{2}_{2}}{(\mu_{2}+a-b)(\mu_{2}+a+b)}
\end{aligned}
\end{equation}
By (\ref{e3.1}) and (\ref{e3.2}) we deduce that $F_{\tau}$ satisfies (\ref{e1.16}) and(\ref{e1.17}). We calculate directly to obtain:
$$\sum^{n}_{i,j=1}\frac{\partial^{2} F_{\tau}}{\partial\lambda_{i}\lambda_{j}}\xi_{i}\xi_{j}=
\sum^{n}_{i=1}(\frac{\xi^{2}_{i}}{(\lambda_{i}+a+b)^{2}}-\frac{\xi^{2}_{i}}{(\lambda_{i}+a-b)^{2}})\leq 0$$
and
$$\sum^{n}_{i,j=1}\frac{\partial^{2} F_{\tau}^{*}}{\partial\lambda_{i}\lambda_{j}}\xi_{i}\xi_{j}=\sum^{n}_{i=1}(\frac{\xi^{2}_{i}}{(\lambda_{i}+(a-b)^{-1})^{2}}-\frac{\xi^{2}_{i}}{(\lambda_{i}+(a+b)^{-1})^{2}})\leq 0.$$
Consequently, $F_{\tau}$ satisfies (\ref{e1.15}). Let $$f_{1}(t)=f_{2}(t)\triangleq n\ln(\frac{t+a-b}{t+a+b}).$$
It is elementary to check that
$$f_{1}(\lambda_{1})\leq F_{\tau}(\lambda_{1},\lambda_{2},\cdots, \lambda_{n})\leq f_{2}(\lambda_{n}).$$
Given $0<\lambda_{1}\leq\lambda_{2}\leq\cdots\leq\lambda_{n}$, if
$$f_{1}(t)\leq F_{\tau}(\lambda_{1},\lambda_{2},\cdots, \lambda_{n}),$$
then we have
\begin{equation}\label{e3.3}
\ln(\frac{t+a-b}{t+a+b})\leq \frac{\sum^{n}_{i=1}\ln(\frac{\lambda_{i}+a-b}{\lambda_{i}+a+b})}{n}.
\end{equation}
It's easy to see that
$$\ln(\frac{t+a-b}{t+a+b})$$
is monotone increasing continuous function on $(0,+\infty)$.
By Lemma \ref{l3.1}, there exist a unique $t_{1} \in [\lambda_{1},\lambda_{n}]$, such that
$$\frac{\sum^{n}_{i=1}\ln(\frac{\lambda_{i}+a-b}{\lambda_{i}+a+b})}{n}=\ln(\frac{t_{1}+a-b}{t_{1}+a+b}).$$
Combining with (\ref{e3.3}), we obtain
\begin{equation*}
\ln(\frac{t+a-b}{t+a+b})\leq \ln(\frac{t_{1}+a-b}{t_{1}+a+b}).
\end{equation*}
This implies $t\leq t_{1}$. Using the same methods, if
$$f_{2}(t)\geq F_{\tau}(\lambda_{1},\lambda_{2},\cdots, \lambda_{n}),$$
then there exist a unique $t_{2} \in [\lambda_{1},\lambda_{n}]$ such that $t\geq t_{2}$. Putting these facts together,
we see $F_{\tau}$ satisfies (\ref{e1.4aa})-(\ref{e1.15b}) for  $\tau\in(0,\frac{\pi}{4})$.\\

Case 2, $\tau\in(\frac{\pi}{4},\frac{\pi}{2})$.

 A direct calculation as in case 1 gives
$$\frac{\partial F_{\tau}}{\partial\lambda_{i}}=\frac{2b}{(\lambda_{i}+a-b)^{2}+(\lambda_{i}+a+b)^{2}}>0.$$
Then the equation (\ref{e2.1}) is parabolic and  $F_{\tau}$ also satisfies (\ref{e1.4aa}).
For any $\mu_{1}>0, \mu_{2}>0$,  if
$\lambda_{1}\leq \mu_{1}, \lambda_{n}\geq \mu_{2},$ then we deduce that
\begin{equation}\label{e3.4}
\begin{aligned}
\frac{2nb}{(a-b)^{2}+(a+b)^{2}}&\geq\sum^{n}_{i=1}\frac{\partial F_{\tau}}{\partial\lambda_{i}}=\sum^{n}_{i=1}\frac{2b}{(\lambda_{i}+a-b)^{2}+(\lambda_{i}+a+b)^{2}}\\
&\geq \frac{2b}{(\lambda_{1}+a-b)^{2}+(\lambda_{1}+a+b)^{2}}\\
&\geq  \frac{2b}{(\mu_{1}+a-b)^{2}+(\mu_{1}+a+b)^{2}}
\end{aligned}
\end{equation}
and
\begin{equation}\label{e3.5}
\begin{aligned}
nb&\geq\sum^{n}_{i=1}\frac{\partial F_{\tau}}{\partial\lambda_{i}}\lambda^{2}_{i}=\sum^{n}_{i=1}\frac{2b\lambda^{2}_{i}}{(\lambda_{i}+a-b)^{2}+(\lambda_{i}+a+b)^{2}}\\
&\geq \frac{2b\lambda^{2}_{n}}{(\lambda_{n}+a-b)^{2}+(\lambda_{n}+a+b)^{2}}\\
&\geq  \frac{2b\mu^{2}_{2}}{(\mu_{2}+a-b)^{2}+(\mu_{2}+a+b)^{2}}.
\end{aligned}
\end{equation}
By (\ref{e3.4}) and (\ref{e3.5}),  we see that $F_{\tau}$ satisfies (\ref{e1.16}) and(\ref{e1.17}). Clearly, we calculate directly to have:
$$\sum^{n}_{i,j=1}\frac{\partial^{2} F_{\tau}}{\partial\lambda_{i}\lambda_{j}}\xi_{i}\xi_{j}=
\sum^{n}_{i=1}-\frac{8(\lambda_{i}+a)b\xi^{2}_{i}}{((\lambda_{i}+a-b)^{2}+(\lambda_{i}+a+b)^{2})^{2}}\leq 0$$
and
$$\sum^{n}_{i,j=1}\frac{\partial^{2} F_{\tau}^{*}}{\partial\lambda_{i}\lambda_{j}}\xi_{i}\xi_{j}=\sum^{n}_{i=1}-\frac{8b(a+\lambda_{i}(a^{2}+b^{2}))
\xi^{2}_{i}}{((1+\lambda_{i}(a-b))^{2}+(1+\lambda_{i}(a+b))^{2})^{2}}\leq 0.$$
Therefore, $F_{\tau}$ satisfies (\ref{e1.15}). We define $$f_{1}(t)=f_{2}(t)\triangleq n\arctan(\frac{t+a-b}{t+a+b}).$$
Note that the above functions are monotone increasing and continuous in $t$.   We  have the pointwise inequalities
$$f_{1}(\lambda_{1})\leq F_{\tau}(\lambda_{1},\lambda_{2},\cdots, \lambda_{n})\leq f_{2}(\lambda_{n}).$$
Given $0<\lambda_{1}\leq\lambda_{2}\leq\cdots\leq\lambda_{n}$, if
$$f_{1}(t)\leq F_{\tau}(\lambda_{1},\lambda_{2},\cdots, \lambda_{n}),$$
then we obtain
\begin{equation}\label{e3.6}
\arctan(\frac{t+a-b}{t+a+b})\leq \frac{\sum^{n}_{i=1}\arctan(\frac{\lambda_{i}+a-b}{\lambda_{i}+a+b})}{n}.
\end{equation}
By Lemma \ref{l3.1}, there exist a unique $t_{1} \in [\lambda_{1},\lambda_{n}]$, such that
$$\frac{\sum^{n}_{i=1}\arctan(\frac{\lambda_{i}+a-b}{\lambda_{i}+a+b})}{n}=\arctan(\frac{t_{1}+a-b}{t_{1}+a+b}).$$
Combining with (\ref{e3.6}), we conclude that
\begin{equation*}
\arctan(\frac{t+a-b}{t+a+b})\leq \arctan(\frac{t_{1}+a-b}{t_{1}+a+b}).
\end{equation*}
This implies $t\leq t_{1}$. Using the same methods, if
$$f_{2}(t)\geq F_{\tau}(\lambda_{1},\lambda_{2},\cdots, \lambda_{n}),$$
then there exists a unique $t_{2} \in [\lambda_{1},\lambda_{n}]$ such that $t\geq t_{2}$. To summarize, we have also shown that
$F_{\tau}$ satisfies (\ref{e1.4aa})-(\ref{e1.15b}) for  $\tau\in(\frac{\pi}{4},\frac{\pi}{2})$.
Finally, combining case 1 with case 2, we obtain the desired results.
\end{proof}

{\bf Proof of Theorem \ref{t1.1}:}
Using Lemma \ref{l3.2}, the operator $F_{\tau}$ satisfies  the hypotheses in (\ref{e1.4aa})-(\ref{e1.15b}). By Proposition \ref{p1.1}, the assertion follows.
\qed

\end{document}